\documentclass[11pt]{article}

\usepackage[margin=1in]{geometry}
\usepackage[T1]{fontenc}
\usepackage[utf8]{inputenc}
\usepackage{amsmath,amssymb,amsthm,mathtools}
\usepackage{graphicx}
\usepackage{listings}
\usepackage{hyperref}

\ifdefined\pdfgentounicode
  \input glyphtounicode
\fi

\hypersetup{
  colorlinks=true,
  linkcolor=black,
  citecolor=black,
  urlcolor=black
}

\newtheorem{theorem}{Theorem}[section]

\newtheorem{lemma}[theorem]{Lemma}
\newtheorem{corollary}[theorem]{Corollary}

\theoremstyle{definition}
\newtheorem{definition}[theorem]{Definition}
\newtheorem{remark}[theorem]{Remark}

\newcommand{\ZZ}{\mathbb{Z}}
\newcommand{\QQ}{\mathbb{Q}}
\newcommand{\RR}{\mathbb{R}}
\newcommand{\SO}{\operatorname{SO}}

\title{A proper Euler magic matrix of order $5$}

\author{Scott Duke Kominers%
\thanks{%
\textsc{Harvard Business School; Department of Economics and Center of
Mathematical Sciences and Applications, Harvard University; and a16z crypto.}
I used LLMs to assist with computations and coding in the preparation of this
article, especially Claude~Opus~4.8 and GPT-5.5~Pro (both accessed in part via
Poe with the support of Quora, where I am an advisor). In particular,
Claude~Opus~4.8 was used to identify the Givens rotation $G(t)$ and solution
values $(x,y)$ applied in Section~\ref{sec:construction}, as well as to produce
the arithmetic verification script presented in Appendix~\ref{app:code}. The
problem, methods, and eventual written form are my own; and of course any
errors remain my responsibility. This work was conducted while I was visiting
the Technological Innovation, Entrepreneurship, and Strategic Management (TIES)
Group at the MIT Sloan School of Management; I greatly appreciate their
hospitality.}}

\date{July 18, 2026}

\begin{document}

\maketitle

\begin{abstract}
An \emph{Euler magic matrix} is an integer matrix $M$ with $MM^{t}=\gamma I$
whose squared entries sum to $\gamma$ along both main diagonals; it is
\emph{proper} if its squared entries are pairwise distinct. Euler constructed an
order-$4$ proper example, and M\"{u}ller settled orders $3$ (none exist) and
$8$, leaving order $5$ as the smallest open case. We construct such a matrix, by
rotating one of M\"{u}ller's ``near-misses'' under a mirror-symmetric coordinate
pair so that the two diagonal conditions collapse to a single rational equation;
the same invariant suggests a uniform approach to the odd orders.
\end{abstract}

\begin{center}
\end{center}

\section{Introduction}

A \emph{magic square of squares} is a square matrix of distinct perfect squares
whose rows, columns, and two main diagonals share a common sum. Whether one
exists of order $3$ is a famous old problem, usually credited to Lucas (1876),
revived by LaBar~\cite{LaBar} in 1984, and popularized with a \$100 prize by
Gardner~\cite{Gardner}; it is Problem~D15 of Guy~\cite{Guy}. Despite intensive
computation and several celebrated near-misses~\cite{Sallows}, no order-$3$
example is known, and geometric heuristics suggest that examples, if any, are
rare~\cite{VarillyAlvarado}. By contrast, magic squares of squares of every
order $n\geq 4$ exist, by a circle-method theorem of Rome and
Yamagishi~\cite{RomeYamagishi}. Explicit small examples and a survey are given
by Boyer~\cite{Boyer1,Boyer2} and Bremner~\cite{Bremner1,Bremner2}; see also
Robertson~\cite{Robertson}.

Euler introduced a powerful way to produce candidates for such squares: If $M$
is an integer matrix with $MM^{t}=\gamma I$ and $\gamma\neq 0$, then $M$ is
invertible, so $M^{t}M=\gamma I$ as well, and the matrix of entrywise squares
$S=(m_{ij}^{2})$ has every row and column sum equal to $\gamma$; only the two
diagonal sums and the distinctness of the entry-squares remain. Euler treated
the order-$4$ case in work presented in 1770~\cite{Euler}, obtaining what may
be interpreted in modern terms as a rationally parametrized family of
$4\times 4$ orthogonal matrices. Euler's construction can be recast using
quaternion multiplication~\cite{Hurwitz,OswaldSteuding}; the analogous
construction at order $8$ uses octonions~\cite{Pirsic1,Pirsic2,Muller}.

Following M\"{u}ller~\cite{Muller}, we make the underlying object precise as
follows.

\begin{definition}\label{def:emm}
Let $R$ be a subring of $\RR$ (in this paper, typically $R=\ZZ$; occasionally,
$R=\QQ$) and let $n\geq 1$. A matrix $M=(m_{ij})\in R^{n\times n}$ is an
\emph{Euler magic matrix of order $n$ over $R$} if, for some
$\gamma\in R\setminus\{0\}$,
\begin{align}
  \textstyle MM^{t} &= \gamma I_n, \label{eq:orth}\\
  \textstyle \sum_{k=1}^{n} m_{kk}^{2} &= \gamma, \label{eq:diag}\\
  \textstyle \sum_{k=1}^{n} m_{k,\,n+1-k}^{2} &= \gamma; \label{eq:anti}
\end{align}
it is \emph{proper} if the $n^{2}$ absolute values $|m_{ij}|$ are pairwise
distinct, equivalently, if the squares $m_{ij}^{2}$ are all distinct. When
$R=\ZZ$, the entrywise square of a proper Euler magic matrix is a magic square
of squares.
\end{definition}

We work over $\ZZ$ throughout. This is equivalent to working over $\QQ$: if
$M\in \QQ^{n\times n}$ is Euler magic and $d$ clears denominators, then
$dM\in \ZZ^{n\times n}$ is Euler magic with constant $d^{2}\gamma$, and
properness is unchanged.

M\"{u}ller~\cite{Muller} recently proved that no Euler magic matrix of order $3$
exists over $\QQ$ because the two diagonal conditions would force
$\sqrt{3}\in\QQ$. He also constructed a proper Euler magic matrix of order $8$
by octonionic multiplication, realizing a proposal of Pirsic~\cite{Pirsic1,Pirsic2},
and showed that Euler magic matrices exist in every order $n\neq 3$, though the
ones he exhibits are far from proper. He left the order-$5$ proper case
explicitly open, exhibiting Euler magic matrices of order $5$ whose
entry-squares take $24$ of the possible $25$ distinct values---one short of
properness---but obtained neither a proper example nor a proof of impossibility.

The distinction between the two problems is essential here. A magic square of
squares need not arise from an orthogonal matrix, so the theorem of Rome and
Yamagishi~\cite{RomeYamagishi} does not directly produce proper Euler magic
matrices. Meanwhile, as orders $1$ and $2$ are essentially automatic,\footnote{%
Order $1$ matrices are trivially proper, while order $2$ admits no proper
example, since for
$M=\bigl(\begin{smallmatrix}a&b\\ c&d\end{smallmatrix}\bigr)$ the conditions
force $a^{2}=c^{2}$ and $b^{2}=d^{2}$, leaving at most two distinct absolute
values.}
order~$5$ was the smallest order for which the existence of proper Euler magic
matrices was unknown. Our main result resolves the question: a proper Euler
magic matrix of order~$5$ exists; we construct one explicitly in
Section~\ref{sec:construction}.

The proof is elementary. Section~\ref{sec:reduction} recalls M\"{u}ller's
reduction, valid in every odd order, of the problem to rational points of a
variety $V_n\subset \SO(n)$. Section~\ref{sec:mirror} isolates the mechanism
behind our construction, namely that rotations in a mirror-symmetric coordinate
pair preserve the sum of the two diagonal conditions, collapsing them to a
single equation. Section~\ref{sec:construction} carries out the construction,
and Section~\ref{sec:remarks} discusses larger odd orders. All displayed
matrices and identities are verified by exact arithmetic (see Appendix).

\section{The reduction for odd order}\label{sec:reduction}

Let $D(M)=\sum_{i} m_{ii}^{2}$ and $A(M)=\sum_{i} m_{i,\,n+1-i}^{2}$
denote the diagonal and anti-diagonal sums of squares, so that
\eqref{eq:diag}--\eqref{eq:anti} read as $D(M)=A(M)=\gamma$. The following
reduction is due to M\"{u}ller~\cite[Lemma~2.1]{Muller}; we include the short
proof for completeness.

\begin{lemma}[M\"{u}ller]\label{lem:reduction}
Let $n$ be odd and let $M\in \ZZ^{n\times n}$ satisfy $MM^{t}=\gamma I_n$ with
$\gamma\neq 0$. Then $\gamma$ is a perfect square, $\gamma=c^{2}$ with
$c\in\ZZ_{>0}$, and $Q:=M/c$ is a rational orthogonal matrix. After replacing
one row of $M$ by its negative if necessary, $Q\in \SO(n,\QQ)$. Moreover, $M$
satisfies \eqref{eq:diag}--\eqref{eq:anti} if and only if
\[
  D(Q)=1
	\qquad\text{and}\qquad
	A(Q)=1;
\]
and $M$ is proper if and only if the $n^{2}$ absolute values $|Q_{ij}|$ are
distinct.
\end{lemma}

\begin{proof}
First, $\gamma>0$, as it equals the sum of squares of any row of $M$, and every
row is nonzero since $MM^{t}=\gamma I_n\neq 0$. Taking determinants gives
$(\det M)^{2}=\gamma^{n}$; as $n$ is odd, each prime divides $\gamma$ to an even
power, so $\gamma=c^{2}$ for some $c\in\ZZ_{>0}$, and $Q=M/c$ satisfies
$QQ^{t}=I_n$. Replacing one row of $M$ by its negative preserves
\eqref{eq:orth}--\eqref{eq:anti} and flips the sign of $\det M$, so the
resulting $Q=M/c$ lies in $\SO(n,\QQ)$. Dividing \eqref{eq:diag} and
\eqref{eq:anti} by $\gamma=c^{2}$ gives $D(Q)=A(Q)=1$, and the equalities
$m_{ij}^{2}=m_{kl}^{2}$ are unaffected by the scaling.
\end{proof}

We write
\[
  V_n=\{Q\in \SO(n): D(Q)=A(Q)=1\}
\]
for the $\QQ$-subvariety of $\SO(n)$ cut out by the quadratic conditions $D=1$
and $A=1$. By Lemma~\ref{lem:reduction}, proper Euler magic matrices of odd
order $n$ correspond, up to a scaling and a row sign change, to the rational
points of $V_n$ lying off the closed subset
\[
  \bigcup_{(i,j)\neq (k,l)} \{Q_{ij}^{2}=Q_{kl}^{2}\}.
\]
In order $3$, the $\sqrt{3}$ obstruction of~\cite{Muller} shows that
$V_3(\QQ)$ is empty; for $n=5$ we exhibit a rational point outside the
coincidence locus.

\section{Mirror rotations and a single-equation reduction}\label{sec:mirror}

Call a pair $\{a,a^{*}\}$ with $a^{*}=n+1-a$ and $a\neq a^{*}$ a
\emph{mirror pair}; it is the orbit of $a$ under the reflection
$i\mapsto n+1-i$ that fixes the anti-diagonal. A \emph{Givens rotation supported
on $\{a,a^{*}\}$} acts as the identity outside rows and columns $a,a^{*}$, where
it is the planar rotation with block
$\begin{psmallmatrix}c&-s\\ s&c\end{psmallmatrix}$, $c^{2}+s^{2}=1$
(see~\cite{Givens}).

\begin{theorem}\label{thm:mirror}
For every $M\in \RR^{n\times n}$, the functional $D+A$ is invariant under left-
or right-multiplication of $M$ by any Givens rotation supported on a mirror pair.
\end{theorem}

\begin{proof}
It suffices to treat left multiplication, as $D$ and $A$ are
transpose-invariant, and if $G$ denotes the rotation then
$(MG)^{t}=G^{t}M^{t}$ reduces the right case to the left. Multiplying on the
left by a Givens rotation on rows $a,a^{*}$ changes only those two rows; the
four entries it mixes that contribute to $D+A$ are the diagonal entries
$u:=M_{aa}$ and $z:=M_{a^{*}a^{*}}$, and the anti-diagonal entries
$v:=M_{a,a^{*}}$ and $w:=M_{a^{*}a}$ (recall that $a^{*}=n+1-a$). After
rotation, the diagonal contribution is $(cu-sw)^{2}+(sv+cz)^{2}$ and the
anti-diagonal contribution is $(cv-sz)^{2}+(su+cw)^{2}$; their sum is
\[
  (cu-sw)^{2}+(su+cw)^{2}+(sv+cz)^{2}+(cv-sz)^{2}
  = u^{2}+w^{2}+v^{2}+z^{2},
\]
as the cross-terms cancel and $c^{2}+s^{2}=1$. The claim then follows since
$u^{2}+w^{2}+v^{2}+z^{2}$ equals exactly the pre-rotation contribution of
 $u$, $w$, $v$, and $z$ to $D+A$, and all other entries are unchanged.
\end{proof}

Note that in Theorem~\ref{thm:mirror} the mirror-pair-support hypothesis is
essential: a rotation supported on a non--mirror pair changes $D+A$ in general.
For instance, left-multiplying $I_5$ by a rotation in rows $1,2$ (not a mirror
pair, as $1^{*}=5$) sends $D+A=6$ to $4+2c^{2}$, which equals $6$ only when
$s=0$, i.e., when the rotation does not mix the two rows.

\begin{corollary}\label{cor:single}
Let $M$ be an Euler magic matrix of order $n$ with constant $\gamma$, and let
$N$ be obtained from $M$ by any product of mirror-pair Givens rotations applied
on the left and/or the right. Then $NN^{t}=\gamma I_n$ and
$D(N)+A(N)=2\gamma$; consequently, $N$ is Euler magic with constant $\gamma$ if
and only if $D(N)=A(N)$.
\end{corollary}

\begin{proof}
Each rotation is orthogonal, so $NN^{t}=\gamma I_n$ holds throughout. By
Theorem~\ref{thm:mirror},
\[
  D(N)+A(N)=D(M)+A(M)=2\gamma.
\]
Given $D(N)+A(N)=2\gamma$, the two conditions $D(N)=\gamma$ and
$A(N)=\gamma$ hold if and only if $D(N)=A(N)$.
\end{proof}

\begin{remark}\label{rem:precedent}
M\"{u}ller~\cite[\S3]{Muller} already works with the sum and difference of the
two diagonal conditions \eqref{eq:diag}--\eqref{eq:anti}, noting that they are
better behaved than the conditions individually. Theorem~\ref{thm:mirror}
explains this for the sum: it is invariant under an entire group of rotations,
which is what collapses the pair of conditions to a single condition in the
sequel.
\end{remark}

\section{The construction}\label{sec:construction}

We begin from an Euler magic matrix observed by M\"{u}ller~\cite[\S4]{Muller}:
\[
  M_0=
  \begin{pmatrix}
    -106 & -32 & -8  & -75 & -50\\
    -4   & -38 & -120& 58  & -35\\
    24   & \mathbf{20} & -73 & -88 & 80\\
    61   & 66  & -16 & -46 & -100\\
    70   & -115& \mathbf{20} & -40 & -18
  \end{pmatrix},
  \qquad
  \gamma_0=143^{2}=20449,
\]
whose entry-squares take exactly $24$ distinct values---a near-miss that fails
to be proper by a single coincidence (indicated in \textbf{bold}).

For a rational parameter $t$, let $G(t)$ be the Givens rotation supported on the
mirror pair $\{2,4\}$, equal to the identity except for
\[
  \begin{pmatrix}
    G(t)_{22} & G(t)_{24}\\
    G(t)_{42} & G(t)_{44}
  \end{pmatrix}
  =
  \begin{pmatrix}
    \dfrac{1-t^{2}}{1+t^{2}} & -\dfrac{2t}{1+t^{2}}\\[8pt]
    \dfrac{2t}{1+t^{2}} & \dfrac{1-t^{2}}{1+t^{2}}
  \end{pmatrix},
\]
and set $N(x,y)=G(x)M_0G(y)$.\footnote{%
This rotation fixes the nine entries whose row and column both lie in
$\{1,3,5\}$, so it resolves the seed's sole coincidence---the entries equal to
$20$ at $(3,2)$ and $(5,3)$---through the mobile entry at $(3,2)$;
Section~\ref{sec:remarks} describes the reach of a mirror pair in general.}
At the two values used below, the rotations are, on the $\{2,4\}$ block,
\[
  G\!\left(\frac{755}{547}\right)
  =
  \frac{1}{434617}
  \begin{pmatrix}
    -135408 & -412985\\
    412985  & -135408
  \end{pmatrix},
  \qquad
  G\!\left(\frac{671}{631}\right)
  =
  \frac{1}{424201}
  \begin{pmatrix}
    -26040 & -423401\\
    423401 & -26040
  \end{pmatrix},
\]
so $d=434617\cdot 424201=184364966017$ clears all denominators.

\begin{theorem}\label{thm:size5}
The integer matrix $M_5=d\,N(755/547,\,671/631)$, shown in full as
\[
\resizebox{\textwidth}{!}{$
M_5=
\begin{pmatrix}
-19542686397802 & -13439137777515 & -1474919728136  & 6737359718344   & -9218248300850\\
-10456746813053 & 5294014476286   & 9695843480720   & 9071853210066   & 19529272313780\\
4424759184408   & -16419868506296 & -13458642519241 & -2684411900500  & 14749197281360\\
-4204607349428  & 13420420317418  & -20103594654072 & 9642581583638   & -387581848675\\
12905547621190  & -6059186828480  & 3687299320340   & 21614683395155  & -3318569388306
\end{pmatrix}$},
\]
is a proper Euler magic matrix of order $5$, with
\[
  \gamma=(143d)^{2}=26364190140431^{2}=695070521760799151500865761
\]
and entry $\gcd$ equal to $1$.
\end{theorem}

\begin{proof}
Since $G(x)$ and $G(y)$ are orthogonal,
$N(x,y)N(x,y)^{t}=\gamma_0 I_5$ for all $x,y$, and $\{2,4\}$ is a mirror pair,
so Corollary~\ref{cor:single} applies: $N(x,y)$ is Euler magic with constant
$\gamma_0$ if and only if $D(N)=A(N)$. A direct computation gives
\[
  D\bigl(N(x,y)\bigr)-A\bigl(N(x,y)\bigr)
  =
  \frac{-256\,F_1(x,y)\,F_2(x,y)}
       {(1+x^{2})^{2}(1+y^{2})^{2}},
\]
where
\[
\begin{aligned}
  F_1 &= 26x^{2}y-5xy^{2}-4xy+5x-26y,\\
  F_2 &= x^{2}y^{2}-5x^{2}y-x^{2}+26xy^{2}-26x-y^{2}+5y+1.
\end{aligned}
\]
Since the denominator is positive for real $x,y$, $N(x,y)$ is Euler magic
exactly when $F_1F_2=0$. The point
$(x_0,y_0)=(755/547,\,671/631)$ was found by a height-bounded search along the
curve $F_2=0$; that it lies on the curve can be checked by hand. Indeed,
\[
  F_2=(x^{2}+26x-1)(y^{2}-1)+5(1-x^{2})y,
\]
so, for $y\ne 0$, dividing by $y$ shows that $F_2=0$ is equivalent to
\[
  (x^{2}+26x-1)(y-y^{-1})=5(x^{2}-1);
\]
at $(x_0,y_0)$ both sides equal $26\cdot 52080/547^{2}$, using
\[
  x_0^{2}+26x_0-1=\frac{26\cdot 423401}{547^{2}},\qquad
  5(x_0^{2}-1)=\frac{26\cdot 52080}{547^{2}},\qquad
  y_0-y_0^{-1}=\frac{52080}{423401}.
\]
Hence $N(755/547,\,671/631)$ is Euler magic with constant
$\gamma_0=143^{2}$. Clearing denominators by $d$ gives the integer matrix
$M_5=d\,N(755/547,\,671/631)$ stated in Theorem~\ref{thm:size5}, Euler magic
with constant $(143d)^{2}=26364190140431^{2}$. Properness is the nonvanishing of
the $\binom{25}{2}=300$ integer differences $M_{ij}^{2}-M_{kl}^{2}$, and the
entries have $\gcd$ $1$; both are exact integer checks. All computational claims
are verified via exact arithmetic in Python; see Appendixes~\ref{app:code}
and~\ref{app:output}.
\end{proof}

\begin{remark}\label{rem:notisolated}
Since $\det M_5=(143d)^{5}>0$ (see Appendix~\ref{app:code}), the rescaled
matrix
\[
  Q=M_5/(143d)=M_5/26364190140431
\]
lies in $\SO(5,\QQ)$, so $Q\in V_5$. At $Q$, the differentials of $D$ and $A$
are linearly independent on the tangent space
$T_Q\SO(5)=\{XQ:X^{t}=-X\}$; the relevant Jacobian has rank $2$ (again, see
Appendix~\ref{app:code}). Hence $V_5$ is smooth of dimension
$\binom{5}{2}-2=8$ near $Q$, so $Q$ is not an isolated point of $V_5$;
equivalently, $M_5$ lies in a positive-dimensional real family of Euler magic
matrices with the same constant, among them the image of the rotation curve
$F_2(x,y)=0$. We neither need nor claim infinitude of rational points for the
theorem.
\end{remark}

\section{Larger odd orders}\label{sec:remarks}

The reduction of Section~\ref{sec:reduction} and the mechanism of
Section~\ref{sec:mirror} are uniform in the odd order $n$. For odd $n$, the
mirror pairs are $\{1,n\},\{2,n-1\},\dots$, with the center $(n+1)/2$ fixed,
and Corollary~\ref{cor:single} reduces the two diagonal conditions, along any
family of mirror-pair rotations applied to an Euler magic seed, to the single
equation $D=A$. Our order-$5$ example is an instance with one mirror pair and a
single near-miss seed.

The ingredient specific to a given odd $n$ is a suitable seed---an Euler magic
matrix of order $n$ whose entry-squares are nearly all distinct---together with
a rational solution of that equation. For $n=7$, for example, the mirror pairs
are $\{1,7\},\{2,6\},\{3,5\}$ with center $4$, and the same template would apply
once such a seed is in hand.

The reach of a single mirror pair is limited in a precise way. The reflection
$i\mapsto n+1-i$ fixes only the center $(n+1)/2$, and its two-element orbits are
the mirror pairs; a rotation on $\{a,a^{*}\}$ can affect only the entries in rows
or columns $a$ or $a^{*}$, and generically affects precisely those entries.
Hence the center entry is the only one immovable by every mirror pair; in any pair
of repeated positions, unless both entries are constrained by special zero patterns,
at least one position is accessible to a suitable mirror-pair rotation, chosen
according to where the entries lie. For the order-$5$ seed we used, rotation on
$\{2,4\}$ moves the entry at $(3,2)$, while $\{1,5\}$ would move the one at
$(5,3)$. What is not automatic is the arithmetic step, i.e., preserving $D=A$
while landing off the coincidence locus.

Beyond order $5$, existence of proper Euler magic matrices in odd orders remains
open, and seems potentially more difficult in general than the corresponding
existence statement for magic squares of squares, which holds for all
$n\geq 4$~\cite{RomeYamagishi}. Condition~\eqref{eq:orth} is rigid, and a magic
square of squares need not come from an orthogonal matrix. The group $\SO(n)$ is
itself rational over $\QQ$, by the Cayley parametrization~\cite{LiebeckOsborne},
so the difficulty lies entirely in the two added quadratic conditions cutting
out $V_n$. Let $C\subset V_n$ be the coincidence locus and
$U_n=V_n\setminus C$; a proper example of order $n$ is exactly a rational point
of $U_n$. Such a point exists whenever some $\QQ$-irreducible component of $V_n$
is not contained in $C$ and is unirational over $\QQ$, for then its rational
points are Zariski-dense in that component and so meet $U_n$. We leave the
existence of such a component---and, more basically, the density of
$V_n(\QQ)$---as open questions for future work.

One could also ask for a more conceptual order-$5$ construction. Orders $4$ and
$8$ are governed by the quaternions and octonions; a possible conceptual source
for order $5$ is the exceptional isomorphism
$\mathrm{Spin}(5)\cong \mathrm{Sp}(2)$, under which $\mathrm{Sp}(2)$ acts on the
five-dimensional space of trace-free quaternion-Hermitian $2\times 2$ matrices
and realizes $\SO(5)$. We do not know whether this $\mathrm{Sp}(2)$-model has
been used for Euler magic matrices or magic squares of squares; it might give a
more systematic approach to the odd orders.

Finally, the Kronecker product of two Euler magic matrices is again Euler magic:
if $M,N$ are Euler magic with constants $\gamma_M,\gamma_N$, then
\[
  (M\otimes N)(M\otimes N)^{t}=\gamma_M\gamma_N I,
\]
while
\[
  D(M\otimes N)=D(M)D(N)
  \qquad\text{and}\qquad
  A(M\otimes N)=A(M)A(N)
\]
supply the two diagonal sums. So order $5$ combines with other orders to give
Euler magic matrices of composite order---although keeping such a product proper
takes a careful choice of factors, in the spirit of product constructions for
magic squares~\cite{ZhangChenLi}.

\appendix

\section{Exact-arithmetic verification}\label{app:code}

The following self-contained script was run under Python~3 with the
\texttt{SymPy} library; all arithmetic is exact over $\ZZ$, $\QQ$,
and the relevant rational function fields. It verifies the seed $M_0$
(including $\det M_0=143^{5}$), the single-equation identity for $D-A$,
the rational point and the integral matrix $M_5$ (including properness,
the constant $\gamma$, and $\det M_5=(143d)^{5}$), the mirror-rotation
identity of Theorem~\ref{thm:mirror}, and the tangent-rank statement of
Remark~\ref{rem:notisolated}. Each test is recorded as a Boolean, and the final
line is reached only if all $15$ pass. The output of the verification script is
presented in Appendix~\ref{app:output}.

\medskip

\begin{lstlisting}[language=Python]
# Exact verification for "A proper Euler magic matrix of order five".
# Python 3 + SymPy; all arithmetic is exact over Z, Q, or Q[x,y].
# Every assertion is collected; the final line prints only if all pass.

from math import gcd
from functools import reduce
import platform
import sympy as sp

print("Python", platform.python_version())
print("SymPy ", sp.__version__)

results = []
def check(label, cond):
    cond = bool(cond)
    results.append(cond)
    print(f"  [{'PASS' if cond else 'FAIL'}] {label}")

def D(M):                       # sum of squares on the main diagonal
    return sum(M[i, i]**2 for i in range(M.rows))
def A(M):                       # sum of squares on the anti-diagonal
    n = M.rows
    return sum(M[i, n-1-i]**2 for i in range(n))
def row_sq_sums(M):
    return [sum(M[i, j]**2 for j in range(M.cols)) for i in range(M.rows)]
def distinct_sq(M):
    return len({sp.simplify(z**2) for z in M})
def is_magic(M, g):
    n = M.rows
    return (M.rows == M.cols and M*M.T == g*sp.eye(n)
            and M.T*M == g*sp.eye(n) and D(M) == g and A(M) == g)
def entry_gcd(M):
    return reduce(gcd, [abs(int(z)) for z in M])
def Givens(n, i, j, t):         # rational Givens rotation, block [[c,-s],[s,c]]
    c, s = (1 - t**2)/(1 + t**2), 2*t/(1 + t**2)
    R = sp.eye(n); R[i, i] = c; R[j, j] = c; R[i, j] = -s; R[j, i] = s
    return R

# ---- data ----
M0 = sp.Matrix([[-106, -32, -8, -75, -50],
                [-4, -38, -120, 58, -35],
                [24, 20, -73, -88, 80],
                [61, 66, -16, -46, -100],
                [70, -115, 20, -40, -18]])
gamma0 = 143**2
M5 = sp.Matrix([[-19542686397802, -13439137777515, -1474919728136, 6737359718344, -9218248300850],
                [-10456746813053, 5294014476286, 9695843480720, 9071853210066, 19529272313780],
                [4424759184408, -16419868506296, -13458642519241, -2684411900500, 14749197281360],
                [-4204607349428, 13420420317418, -20103594654072, 9642581583638, -387581848675],
                [12905547621190, -6059186828480, 3687299320340, 21614683395155, -3318569388306]])
c5 = sp.Integer(26364190140431)
gamma5 = c5**2
d = 434617*424201
x, y = sp.symbols("x y")
F1 = 26*x**2*y - 5*x*y**2 - 4*x*y + 5*x - 26*y
F2 = x**2*y**2 - 5*x**2*y - x**2 + 26*x*y**2 - 26*x - y**2 + 5*y + 1
den = (1 + x**2)**2*(1 + y**2)**2
N = Givens(5, 1, 3, x)*M0*Givens(5, 1, 3, y)
x0, y0 = sp.Rational(755, 547), sp.Rational(671, 631)
N0 = Givens(5, 1, 3, x0)*M0*Givens(5, 1, 3, y0)

# ---- seed M0 ----
print("=== seed M0  (gamma0 = 143^2 =", gamma0, ") ===")
print("  row square sums:", row_sq_sums(M0), "  distinct entry-squares:", distinct_sq(M0), "/25")
check("M0 is Euler magic with constant 143^2", is_magic(M0, gamma0))
check("M0 has exactly 24 distinct entry-squares", distinct_sq(M0) == 24)
check("det(M0) = 143^5   (so M0/143 in SO(5,Q))", M0.det() == 143**5)

# ---- single-equation reduction, mirror pair {2,4} (0-based {1,3}) ----
print("=== single-equation reduction (mirror pair {2,4}) ===")
check("D(N)+A(N) = 2*gamma0   (mirror invariance)", sp.simplify(D(N) + A(N) - 2*gamma0) == 0)
check("(D(N)-A(N))*den = -256*F1*F2", sp.simplify((D(N) - A(N))*den + 256*F1*F2) == 0)
check("F2(755/547, 671/631) = 0", F2.subs({x: x0, y: y0}) == 0)

# ---- integral matrix M5 = d * N(755/547, 671/631) ----
print("=== integral matrix M5 = d*N(755/547,671/631),  d =", d, "===")
check("d * N(755/547,671/631) = M5", d*N0 == M5)
check("(143*d)^2 = gamma5", (143*d)**2 == gamma5)
print("  gamma5 =", gamma5)
print("  M5 row square sums:", row_sq_sums(M5), "  distinct entry-squares:", distinct_sq(M5), "/25")
check("M5 is Euler magic with constant gamma5", is_magic(M5, gamma5))
check("M5 is proper (25 distinct entry-squares)", distinct_sq(M5) == 25)
check("M5 entries have gcd 1", entry_gcd(M5) == 1)
check("det(M5) = (143*d)^5   (so M5/(143d) in SO(5,Q))", M5.det() == (143*d)**5)

# ---- mirror-rotation invariance, symbolic, general 5x5 ----
print("=== mirror-rotation invariance, symbolic on a general 5x5 ===")
q = sp.symbols("q0:25"); Qsym = sp.Matrix(5, 5, q); t = sp.symbols("t"); G = Givens(5, 1, 3, t)
check("D(G*Q)+A(G*Q) = D(Q)+A(Q)   (left)",  sp.factor(D(G*Qsym) + A(G*Qsym) - D(Qsym) - A(Qsym)) == 0)
check("D(Q*G)+A(Q*G) = D(Q)+A(Q)   (right)", sp.factor(D(Qsym*G) + A(Qsym*G) - D(Qsym) - A(Qsym)) == 0)

# ---- local dimension of V_5 at Q = M5/(143d) ----
print("=== local dimension of V_5 at Q = M5/(143d) ===")
Q5 = M5/c5
tan, X = [], sp.zeros(5)
for i in range(5):
    for j in range(i + 1, 5):
        v = sp.symbols(f"u{i}{j}"); tan.append(v); X[i, j], X[j, i] = v, -v
Yt = X*Q5
linD = 2*sum(Q5[i, i]*Yt[i, i] for i in range(5))
linA = 2*sum(Q5[i, 4 - i]*Yt[i, 4 - i] for i in range(5))
Jac = sp.Matrix([[sp.diff(linD, v) for v in tan], [sp.diff(linA, v) for v in tan]])
r = Jac.rank()
print("  dim SO(5) = 10,  rank of (dD,dA) at Q =", r, ",  local dim V_5 =", 10 - r)
check("(dD,dA) has rank 2 at Q  (V_5 smooth of dim 8 near Q)", r == 2)

assert all(results), "SOME CHECKS FAILED"
print()
print(f"ALL {len(results)} CHECKS PASSED")
\end{lstlisting}

\section{Verification output}\label{app:output}

\begin{lstlisting}
Python 3.12.3
SymPy  1.14.0
=== seed M0  (gamma0 = 143^2 = 20449 ) ===
  row square sums: [20449, 20449, 20449, 20449, 20449]   distinct entry-squares: 24 /25
  [PASS] M0 is Euler magic with constant 143^2
  [PASS] M0 has exactly 24 distinct entry-squares
  [PASS] det(M0) = 143^5   (so M0/143 in SO(5,Q))
=== single-equation reduction (mirror pair {2,4}) ===
  [PASS] D(N)+A(N) = 2*gamma0   (mirror invariance)
  [PASS] (D(N)-A(N))*den = -256*F1*F2
  [PASS] F2(755/547, 671/631) = 0
=== integral matrix M5 = d*N(755/547,671/631),  d = 184364966017 ===
  [PASS] d * N(755/547,671/631) = M5
  [PASS] (143*d)^2 = gamma5
  gamma5 = 695070521760799151500865761
  M5 row square sums: [695070521760799151500865761, 695070521760799151500865761, 695070521760799151500865761, 695070521760799151500865761, 695070521760799151500865761]   distinct entry-squares: 25 /25
  [PASS] M5 is Euler magic with constant gamma5
  [PASS] M5 is proper (25 distinct entry-squares)
  [PASS] M5 entries have gcd 1
  [PASS] det(M5) = (143*d)^5   (so M5/(143d) in SO(5,Q))
=== mirror-rotation invariance, symbolic on a general 5x5 ===
  [PASS] D(G*Q)+A(G*Q) = D(Q)+A(Q)   (left)
  [PASS] D(Q*G)+A(Q*G) = D(Q)+A(Q)   (right)
=== local dimension of V_5 at Q = M5/(143d) ===
  dim SO(5) = 10,  rank of (dD,dA) at Q = 2 ,  local dim V_5 = 8
  [PASS] (dD,dA) has rank 2 at Q  (V_5 smooth of dim 8 near Q)

ALL 15 CHECKS PASSED
\end{lstlisting}

\providecommand{\bysame}{\leavevmode\hbox to3em{\hrulefill}\thinspace}
\providecommand{\MR}{\relax\ifhmode\unskip\space\fi MR }
\providecommand{\MRhref}[2]{%
  \href{http://www.ams.org/mathscinet-getitem?mr=#1}{#2}%
}
\providecommand{\href}[2]{#2}

\end{document}